\documentclass[12pt]{amsart}

\usepackage{amsmath,amssymb,amsbsy,amsfonts,amsthm,latexsym, old-arrows,
	amsopn,amstext,amsxtra,euscript,amscd,mathrsfs}
\usepackage{amsmath,amssymb,amsbsy,amsfonts,latexsym,amsopn,amstext,cite,
	amsxtra,euscript,amscd,bm}
\usepackage{mathtools}
\usepackage{todonotes}
\usepackage{url}
\usepackage[colorlinks,linkcolor=blue,anchorcolor=blue,citecolor=blue,backref=page]{hyperref}

\renewcommand*{\backref}[1]{}
\renewcommand*{\backrefalt}[4]{%
	\ifcase #1 (Not cited.)%
	\or        (p.\,#2)%
	\else      (pp.\,#2)%
	\fi}

\begin{document}

\newtheorem{counter}{Theorem}
\newcommand{\jed}{\mbox{\boldmath$1$}}
\newcommand{\To}{\longrightarrow}
\newcommand{\A}{\mathcal{A}}
\newcommand{\C}{\mathbb{C}}
\renewcommand{\S}{\mathcal{S}}
\newcommand{\F}{\mathcal{F}}
\newcommand{\R}{\mathbb{R}}
\newcommand{\E}{\mathbb{E}}
\newcommand{\1}{\mathbb{1}}
\renewcommand{\P}{\mathbb{P}}
\newcommand{\Z}{\mathbb{Z}}
\newcommand{\N}{\mathbb{N}}
\newcommand{\W}{\mathbb{W}}
\newcommand{\T}{\mathbb{T}}
\newcommand{\M}{\mathcal{M}}
\newcommand{\U}{\mathcal{U}}
\newcommand{\V}{\mathcal{V}}
\newcommand{\Q}{\mathcal{Q}}
\renewcommand{\S}{\mathcal{S}}
\newcommand{\B}{\mathcal{B}}
\newcommand{\ce}{\mathbb E}
\newcommand{\I}{\mathcal{I}}
\newcommand{\setdef}{\stackrel {\rm {def}}{=}}
\newcommand{\rea}{\text{Re}}
\newcommand{\ds}{\displaystyle}
\newcommand{\script}{\scriptstyle}
\newcommand{\tif}{\tilde f}
\newcommand{\tig}{\tilde g}
\newtheorem*{thank}{\ \ \ \textbf{Acknowledgment}}
\newcounter{tictac}
\newenvironment{fleuveA}{
	\begin{list}{$(\emph{\textbf{A\arabic{tictac}}})$}{\usecounter{tictac}
			\leftmargin 1cm\labelwidth 2em}}{\end{list}}
\def\1{\,\rlap{\mbox{\small\rm 1}}\kern.15em 1}
\def\ind#1{\1_{#1}}
\def\build#1_#2^#3{\mathrel{\mathop{\kern 0pt#1}\limits_{#2}^{#3}}}
\def\tend#1#2{\build\hbox to 12mm{\rightarrowfill}_{#1\rightarrow #2}^{ }}
\def\cor#1{\build\longlefrightarrow_{}^{#1}}
\def\tendn{\tend{n}{\infty}}
\def\converge#1#2#3#4{\build\hbox to
	#1mm{\rightarrowfill}_{#2\rightarrow #3}^{\hbox{\scriptsize #4}}}


\newcommand{\printdate}{\today}
\theoremstyle{definition}

\newtheorem{thm}{Theorem}[section]
\newtheorem{prop}[thm]{Proposition}
\newtheorem{lem}[thm]{Lemma}
\newtheorem{defn}[thm]{Definition}
\newtheorem{notation}[thm]{Notation}
\newtheorem{example}[thm]{Example}
\newtheorem{conj}[thm]{Conjecture}
\newtheorem{prob}[thm]{Problem}
\newtheorem{Prop}[thm]{Proposition}
\newtheorem{que}[thm]{Question}
\newtheorem{rem}[thm]{Remark}
\newtheorem{rems}[thm]{Remarks}

\newtheorem{Cor}[thm]{Corollary}
\newtheorem{thmnonumber}{Theorem}
\newtheorem{fact}[thmnonumber]{Fact}
\newcommand{\cb}{{\mathcal B}}
\newcommand{\ca}{{\mathcal A}}
\newcommand{\cc}{{\mathcal C}}
\newcommand{\cd}{{\mathcal D}}
\newcommand{\cf}{{\mathcal F}}
\newcommand{\ch}{{\mathcal H}}
\newcommand{\cm}{{\mathcal M}}
\newcommand{\hatca}{\widehat{\mathcal A}}
\newcommand{\hatcb}{\widehat{\mathcal B}}
\newcommand{\hatcc}{\widehat{\mathcal C}}
\newcommand{\hatcd}{\widehat{\mathcal D}}
\newcommand{\ot}{\otimes}
\newcommand{\la}{\lambda}
\newcommand{\tf}{T_{\varphi}}
\newcommand{\va}{\varphi}
\newcommand{\vep}{\varepsilon}
\newcommand{\ov}{\overline}
\newcommand{\un}{\underline}
\newcommand{\ux}{\underline{X}}
\newcommand{\uy}{\underline{Y}}
\newcommand{\hatuy}{\widehat{\underline{Y}}}
\newcommand{\hatux}{\widehat{\underline{X}}}
\newcommand{\uz}{\underline{Z}}
\newcommand{\ka}{\mathcal{K}}
\newcommand{\si}{\sigma}
\newcommand{\bmu}{\bm \mu}
\newcommand{\bmnu}{\bm \nu}
\newcommand{\bml}{\bm \lambda}


\newtheorem*{xrem}{Remark}
\newtheorem*{xconj}{Conjecture}
\newtheorem*{xques}{Question}
\newcommand{\intBohr}{\mathop{\mathlarger{\mathlarger{\mathlarger{\landupint}}}}}
\newcommand{\Prod}{\mathop{\mathlarger{\mathlarger{\mathlarger{\prod}}}}}
\newcommand{\Sum}{\mathlarger{\mathlarger{\sum}}}
\def\Max#1{\mathlarger{\max_{#1}}}
\def\Min#1{\mathlarger{\min_{#1}}}
\def\Sup#1{\mathlarger{\sup_{#1}}}
\def\En#1{(\lfloor #1 \rfloor)}

\def \d{{\rm d}}
\def \f{\overline{f}}
\def \g{\overline{g}}
\def \e{\varepsilon}
\def\InvFourier#1{\stackrel{\vee}{#1}}
\def\Invstar#1{\widetilde{#1}}
\def \Cov{\text{Cov}}
\def\Pr{\mathbb P}
\newcommand{\beq}{\begin{equation}}
\newcommand{\eeq}{\end{equation}}
\newcommand{\xbm}{(X,{\mathcal B},\mu)}
\newcommand{\xbmt}{(X,{\mathcal B},\mu,T)}
\newcommand{\tilxbm}{(\tilde{X},\tilde{{\cal B}},\tilde{\mu})}
\newcommand{\zdr}{(Z,{\mathcal D},\rho)}
\newcommand{\ycn}{(Y,{\mathcal C},\nu)}
\title[On the polynomials  homogeneous ergodic $\cdots$"]{ On the polynomials homogeneous ergodic bilinear averages with Liouville and M\"{o}bius weights}
\author{E. H. el Abdalaoui}





\begin{abstract} We establish a generalization of Bourgain double recurrence theorem by proving that for any map $T$  acting on a probability space $(X,\mathcal{A},\mu)$, and for any non-constant polynomials $P,Q$ mapping natural numbers to themselves,  
	 for any $f,g \in L^2(X)$, and for almost all $x \in X$, we have
	
	\[
	\frac{1}{N}¶
	\sum_{n=1}^{N}\boldsymbol{\nu}(n)¶
	f(T^{P(n)}x)g(T^{Q(n)}x) \underset{N¶
		\rightarrow{+\infty}} {\mathrel{\hbox to 12mm{\rightarrowfill}}}0,¶
	\]
	
	where $\boldsymbol{\nu}$ is the Liouville function or the M\"{o}bius¶
	function.  
\end{abstract}

\maketitle
\section{Introduction}
The purpose of this paper is to generalize Bourgain double recurrence theorem (BDRT) proved in \cite{Bourgain-D}. The original proof of this theorem seems to be one of the more difficult proofs in ergodic theory. It used the $\lambda$-separated lemma based on L\'epingle's Lemma from the theory of Martingale as in the proof of Bourgain polynomial ergodic theorem \cite{Bourgain-IHP} (see also \cite{Thouvenot}). But, therein the difficulty lies on some uniformity with respect to the orbit (see \cite[section 4]{Bourgain-D}). Later, Demeter provided an alternative proofs in \cite{Demeter}, and Lacey \cite{Lacey} extended  bilinear  maximal inequality into
$L^{p}$  with  $2/3<p<1 .$ Besides, I. Assani proved that for the  special cases of weak mixing maps with singular spectrum on its Pinsker algebra, the F\"{u}rstenberg's multilinear ergodic average converge almost surely \cite{Assani}. He also noticed that the proof of (BDRT) can be made more accessible for the case of Wiener-Wintner systems \cite{AssaniW}. Subsequently, the author proved that there is a subsequence for which the convergence a.e. of the ergodic multilinear averages holds \cite{elabdal-F}. Very recently, the author provided a simultaneously simple proof of Birkhoff ergodic theorem and Bourgain  homogeneous ergodic bilinear theorem with an extension to polynomials and polynomials in primes \cite{elabdal-poly}.   Here, we focus on the almost everywhere convergence of the homogeneous ergodic bilinear averages with weight. This was initiated by  I. Assani, D. Duncan, and R. Moore in \cite{AssaniDM}. Therein, the authors proved a Wiener-Wintner version of BDRT, that is, the exponential sequences $(e^{2\pi i nt})_{n \in \Z}$ are good weight for the homogeneous ergodic bilinear averages. 
Subsequently, I. Assani and R. Moore  showed that the polynomials exponential sequences $\big(e^{2\pi i P(n)}\big)_{n \in \Z}$  are also uniformly good weights for the homogeneous ergodic bilinear averages
\cite{AssaniM}. One year later, I. Assani \cite{Assani-Nil} and P. Zorin-Kranich \cite{Zorich} proved independently that the nilsequences are
uniformly good weights for the homogeneous ergodic bilinear averages. But, their proofs depend on Bourgain's theorem. Very recently, the author extended Bourgain-Sarnak theorem by proving that the M\"{o}bius and  Liouville functions are a good weight for the homogeneous ergodic bilinear averages \cite{IJL}. But there is a gap in the proof. Here, we will generalize that theorem to the polynomial cases and we will fill the gap. Our proof follows closely the oscillation method of Bourgain combined with the Calder\'{o}n transference principal. For a nice account on this method, we refer to \cite{Nair}, \cite{Thouvenot}. Despite the fact that the classical spectral analysis can not be applied to study the  F\"{u}rstenberg's multilinear ergodic average, we develop here  a spectral  analysis tool based on the Fourier transform for its studies. This is accomplished by applying Cald\'eron principal and the discrete Fourier transform which can be seen as a spectral isomorphism. This point was raised and developed in \cite{elabdal-poly}. We notice that in this setting, the dynamics on the diagonal in F\"{u}rstenberg's ergodic average is interpreted as an operation on the kernels which we introduced here (see equation \eqref{kernel}). The kernel is an average mass on the particles $X=\big\{x_1,\cdots,x_N \big\}$ and the operation is on the diagonal of $X \times X$  denoted by $\odot$. We stress, as it was noticed \cite{elabdal-poly}, that the product on the observable functions in $\ell^2$ functions is interpreted as a convolution. It follows that Babenko–Beckner inequalities come into play.We believe that our tool combined with harmonic analyis methods can be useful to address the problem of the convergence almost everywhere of F\"{u}rstenberg's multilinear ergodic average.\\ 

We recall that the problem of the convergence almost everywhere (a.e.) of the ergodic  multilinear averages was initiated by F\"{u}rstenberg in \cite[Question 1 p.96]{Fbook}. Bourgain answered that question by proving the following:\\

Let $T$ be a map acting on a probability space $(X,\mathcal{A},\nu)$, and $a,b \in\mathbb{Z}$,¶
then for any $f,g \in L^\infty(X)$, the averages

\[
\frac{1}{N}¶
\sum_{n=1}^{N}
f(T^{an}x)g(T^{bn}x)
\]
converge for almost every $x$.

\section{Set up and Tools}
The Liouville function is defined for the positive integers $n$ by
$$
\bml(n)=(-1)^{\Omega(n)},
$$
where $\Omega(n)$ is the length of the word $n$ is the alphabet of prime, that is, $\Omega(n)$ is the number of prime factors of $n$ counted with multiplicities. The M\"{o}bius function is given by
\begin{equation}\label{Mobius}
\bmu(n)= \begin{cases}
1 {\rm {~if~}} n=1; \\
\bml(n)  {\rm {~if~}} n
{\rm {~is~the~product~of~}} r {\rm {~distinct~primes}}; \\
0  {\rm {~if~not}}
\end{cases}
\end{equation}
These two functions are of great importance in number theory since the Prime Number Theorem is equivalent to \begin{equation}\label{E:la}
\sum_{n\leq N}\bml(n)={\rm o}(N)=\sum_{n\leq N}\bmu(n).
\end{equation}
Furthermore, there is a connection between these two functions and Riemann zeta function, namely
$$
\frac1{\zeta(s)}=\sum_{n=1}^{\infty}\frac{\bmu(n)}{n^s} \text{ for any }s\in\mathbb{C}\text{ with }\rea(s)>1.
$$
Moreover, Littlewood proved that the estimate
\[
\left|\ds \sum_{n=1}^{x}\bmu(n)\right|=O\left(x^{\frac12+\varepsilon}\right)\qquad
{\rm as} \quad  x \longrightarrow +\infty,\quad \forall \varepsilon >0
\]
is equivalent to the Riemann Hypothesis (RH) (\cite[pp.315]{Titchmarsh}).\\

We recall that the proof of Sarnak-Bourgain theorem \cite{Sarnak} is based on the following Davenport-Hua's estimation \cite{Da},  \cite[Theorem 10.]{Hua}: for each $A>0$, for any $k \geq 1$, we have
\begin{equation}\label{vin}
\max_{z \in \T}\left|\displaystyle\sum_{n \leq N}z^{n^k}\bmu(n)\right|\leq C_A\frac{N}{\log^{A}N}\text{ for some }C_A>0.
\end{equation}
We refer to \cite{elabdalDCDS} and \cite{Cuny-Weber} for this proof.  Here, we will need the following extension due to Green-Tao in the nilsequences setting. We refer to Theorem 1.1 in \cite{GreenT} for the complete statement  
and for the definition of the nilsequences. Here, precisely,  we need only the following corollary.

\begin{lem}\label{GTao}Let $P$ be a non-constant polynomial mapping natural numbers to themselves. Then, for any $N \geq 1$, we have
\begin{equation}\label{GT}
\max_{\theta \in [0,2\pi)}\left|\frac{1}{N}\displaystyle\sum_{n \leq N}\bmu(n) e^{i P(n) \theta}\right|\leq \frac{C_A}{\log^{A}N}\text{ for some }C_A>0.
\end{equation}	
\end{lem}

\noindent The inequalities \eqref{vin} and \eqref{GT} can be established also for the M\"{o}bius function by applying carefully the following identity:\\

$$\bml(n)=\sum_{d:d^2 | n}\bmu\Big(\frac{n}{d^2}\Big).$$

\section{Some tools on the oscillation method and Calder\'{o}n transference principle}
 Let $k \geq 2$ and $(X,\ca,T_i,\mu)_{i=1}^{k}$ be a family of dynamical systems, that is, for each $i=1,\cdots, k$, $T_i$ is a measure-preserving transformation, ($\forall A \in \ca$, $\mu(T_i^{-1}A)=\mu(A)$.). The sequence of complex number $(a_n)$ is said to be good weight in $L^{p_i}(X,\mu)$, $p_i\geq1$, $i=1,\cdots,k$, with 
 $\sum_{i=1}^{k}\frac1{p_i}=1,$ if, for any $f_i \in L^{p_i}(X,\mu)$, $i=1,\cdots,k$,
 the ergodic $k$-multilinear averages   
 $$\frac1{N}\sum_{j=1}^{N}a_j\prod_{i=1}^{k}f_i(T_i^jx)$$
 converges a.e.. The maximal multilinear ergodic inequality is said to hold in $L^{p_i}(X,\mu)$, $p_i\geq1$, $i=1,\cdots,k$, with 
 $\sum_{i=1}^{k}\frac1{p_i}=1,$ if, for any $f_i \in L^{p_i}(X,\mu)$, $i=1,\cdots,k$, the maximal function given by 
 $$M(f_1,\cdots,f_k)(x)=\sup_{N \geq 1}\Big|\frac{1}{N}\sum_{j=1}^{N}a_j\prod_{i=1}^{k}f_i(T_i^jx)\Big|$$
 satisfy the weak-type inequality 
 $$\sup_{\lambda > 0}\Bigg(\lambda \mu\Big\{x~~:~~M(f)(x)>\lambda \Big\}\Bigg) \leq C \prod_{i=1}^{k}\big\|f_i\big\|_{p_i},$$
 where $C$ is an absolutely constant.\\
 
 As far as the author know, it is seems that it is not known whether the classical maximal multilinear ergodic inequality ($a_n=1$, for each $n$) holds 
 for the general case $n \geq 3$. Nevertheless, we have the following Calder\'{o}n transference principal in the homogeneous case. 
 
\begin{Prop}\label{CalderonP} Let $(a_n)$ be a sequence of complex number and assume that for any $\phi,\psi \in \ell^2(\Z)$, for any non-constant polynomials $P,Q$ mapping natural numbers to themselves,  for any $1 \leq p,q,r \leq +\infty$  such that $\frac{1}{r}=\frac{1}{p}+\frac{1}{q}$,  we have
	\begin{align*}
			\Big\|\sup_{ N  \geq 1}\Big|\frac1{N}\sum_{n=1}^{N}a_n \phi(j+P(n))\psi(j+Q(n)) 
		\Big|\Big\|_{\ell^r(\Z)} \\\leq  C.
		\big\|\phi\big\|_{\ell^p(\Z)}\big\|\psi\big\|_{\ell^q(\Z)},
	\end{align*}

	where $C$ is an absolutely constant. Then, for any dynamical system $(X,\A,T,\mu)$, for any $f \in L^p(X,\mu)$ and $g \in L^q(X,\mu)$ , we have 
		\begin{align*}
	\Big\|\sup_{N \geq 1}\Big|\frac1{N}\sum_{n=1}^{N}a_n f(T^{P(n)}x)g(T^{Q(n)}x) \Big|\Big\|_1\\ \leq
	C \big\|f\big\|_{p}\big\|g\big\|_{q}.
	\end{align*}
\end{Prop}
We further have
\begin{Prop}\label{CalderonP2} Let $(a_n)$ be a sequence of complex number and  assume that for any $\phi,\psi \in \ell^2(\Z)$, for any $\lambda>0$, for any integer $J \geq 2$, for any non-constant polynomials $P,Q$ mapping natural numbers to themselves, for any $1 \leq p,q\leq +\infty$  such that $\frac{1}{p}+\frac{1}{q}=1$, , we have
\begin{align*}
	\sup_{\lambda > 0} \Big(\lambda\Big|\Big\{1 \leq j \leq J~~:~~ \sup_{ N  \geq 1}\Big|\frac1{N}\sum_{n=1}^{N}a_n \phi(j+P(n))\psi(j+Q(n)) 
	\Big|> \lambda \Big\}\Big|\Big) \\\leq C \big\|\phi\big\|_{\ell^p(\Z)}\big\|\psi\big\|_{\ell^q(\Z)},
\end{align*}
	where $C$ is an absolutely constant. Then, for any dynamical system $(X,\A,T,\mu)$, for any $f,g \in L^2(X,\mu)$, we have 
\begin{align*}	
	\sup_{\lambda> 0} \Big(\mu\Big\{x \in X~~:~~\sup_{N \geq 1}\Big|\frac1{N}\sum_{n=1}^{N}a_n f(T^{P(n)}x)g(T^{Q(n)}x) \Big| > \lambda \Big\}\Big)\\ \leq  C
	\big\|f\big\|_{2}.\big\|g\big\|_{2}.
\end{align*} 
\end{Prop}

\noindent It is easy to check that Proposition \ref{CalderonP} and \ref{CalderonP2} hold for the homogeneous $k$-multilinear ergodic averages, for any 
$k \geq 3$. Moreover, it is easy to state and to prove the finitary version where $\Z$ is replaced by $\Z/J\Z$ and the functions $\phi$ and 
$\psi$ with $J$-periodic functions. 
\section{Main result and its proof}
The subject of this section is to state and to prove the main result of this paper and its consequences. We begin by stating our main result. 

\begin{thm}\label{Mainofmain}Let $(X,\ca,\mu,T)$ be an ergodic dynamical system, let $P,Q$ be a non-constant polynomials mapping natural numbers to themselves. Then, for any $f,g \in L^2(X)$, for almost all $x \in X$,
	$$\frac1{N}\sum_{n=1}^{N}\bmnu(n) f(T^{P(n)}x)g(T^{Q(n)}x) \tend{N}{+\infty}0,$$
	where $\bmnu$ is the Liouville function or the M\"{o}bius function.
\end{thm}

\noindent Consequently, we obtain the following theorem 

\begin{Cor}[\cite{IJL}]Let $(X,\ca,\mu,T)$ be an ergodic dynamical system, and $T_1,T_2$ be powers of $T$. Then, for any $f,g \in L^2(X)$, for almost all $x \in X$,
	$$\frac1{N}\sum_{n=1}^{N}\bmnu(n) f(T_1^nx)g(T_2^nx) \tend{N}{+\infty}0,$$
	where $\bmnu$ is the Liouville function or the M\"{o}bius function.
\end{Cor}

Before proceeding to the proof of Theorem \eqref{Mainofmain}, we recall the following notations.\\

Let $T$ be a map acting on a probability space $(X,\mathcal{A},\mu)$ and for any any $\rho>1$, we will denote by $I_\rho$ the set $\Big\{\En{\rho^n}, n \in \N \Big\}.$ The maximal functions are defined by \\
\begin{align*}
&M_{N_0,\bar{N}}(f,g)(x)=\\
&\sup_{\overset{ N_0 \leq N \leq \bar{N}}{N \in I_\rho}}\Big|\frac1{N}\sum_{n=1}^{N}\bmnu(n) f(T^{n}x) g(T^{-n}x)
-\frac1{N_0}\sum_{n=1}^{N_0}\bmnu(n) f(T^{n}x) g(T^{-n}x)\Big|,
\end{align*}
and
\begin{align*}
&M_{N_0}(f,g)(x)=\\
&\sup_{\overset{N \geq N_0}{N \in I_\rho}}\Big|\frac1{N}\sum_{n=1}^{N}\bmnu(n) f(T^{n}x) g(T^{-n}x)-
\frac1{N_0}\sum_{n=1}^{N_0}\bmnu(n) f(T^{n}x) g(T^{-n}x)\Big|.
\end{align*}

Obviously, 
$$\lim_{\bar{N} \longrightarrow +\infty}M_{N_0,\bar{N}}(f,g)(x)=M_{N_0}(f,g)(x).$$

\noindent For the shift $\Z$-action, the maximal functions are denoted by $m_{N_0,\bar{N}}(\phi,\psi)$ and $m_{N_0}(\phi,\psi)$.\\

\noindent At this point, we restate the key theorem in the proof of our main result.  
\begin{thm}\label{CalderonI}For any $\rho>1$, 
	for any $f,g \in \ell^4(\Z)$, 
	for any $K \geq 1$, we have
	\begin{eqnarray}\label{BourgainMaximal1} 
	\sum_{k=1}^{K}
	\Big\|m_{N_k,N_{k+1}}(f,g)\Big\|_{\ell^2(\Z)} < C. \sqrt{K}.
	\big\|f\big\|_{\ell^4(\Z)}\big\|g\big\|_{\ell^4(\Z)},
	\end{eqnarray}
	where $\bmnu$ is the Liouville function or the M\"{o}bius function and $C$ is an absolutely constant which depend only on $\rho$.
\end{thm}

\begin{proof}We proceed by using the finitary method. For that, let 
	$J$ be a large integer, $f,g \in \ell^2(\Z)$ and $p \geq 1$.  We put 
	$$\Z_{J}= \Z/J\Z,$$
	$$\E_{\Z_{J}}(f)=\frac{1}{J}\sum_{j=1}^{J}f(j),$$
	$$\big\|f\big\|_{\ell^p(\Z_{J})}=\Big(\frac{1}{J}\sum_{j=1}^{J}|f(j)|^p\Big)^{\frac1{p}}.$$
	Moreover, as customary , we will denote by $\F$ the discrete Fourier transform on $\Z_{J}$.  
	We recall that
	$$\ds \F(f)(\chi)=\frac{1}{J}\sum_{n \in \Z_{J}}f(n)\chi(-n),$$ for $\chi \in {\widehat{\Z_{J}}}$, we still denote by $f$ the $J$-periodic function associated to $f$. 
	We further have, by the inverse Fourier formula, the following
	$$
	f(j)=\sum_{\chi \in {\widehat{\Z_{J}}}} \F(f)(\chi)\chi(j).
	$$
	We further identify ${\widehat{\Z_{J}}}$ with $\Z_J$ and we put 
	$$\chi_k(n)=e^{2 i \pi  \frac{kn}{J}},\, \, k=0,\cdots, J-1.$$
	\noindent Obviously, we have
	\begin{align}\label{conv}
	&\frac1{N}\sum_{n=1}^{N}\bmnu(n) f(j+P(n))g(j+Q(n))\notag \\
	&= \sum_{k,l=0}^{J-1}
	 \F(f)(\chi_k)  \F(g)(\chi_l) \Big(\frac{1}{N} \sum_{n=1}^{N}\bmnu(n)\chi_{k}(P(n) \chi_{l}(Q(n))\Big)\chi_{k+l}(j) ,
	\end{align}
	Put 
	$$D_{N,k,l}=\frac{1}{N} \sum_{n=1}^{N}\bmnu(n)\chi_{k}(P(n) \chi_{l}(Q(n)).$$
	Therefore, by Lemma \eqref{GTao}, for each $A>0$,  we have
	\begin{eqnarray}\label{DH}
		\Big|D_{N,k,l}\Big| \leq \frac{C_A}{{\big(\log(N)\big)^{A}}},~~~~ \forall k,l \in \Z.
	\end{eqnarray}
 	Moreover, a straightforward computation gives
 \begin{align}\label{square}
 	&\Big|\frac1{N}\sum_{n=1}^{N}\bmnu(n) f(j+P(n))g(j+Q(n))\Big|^2 \notag\\
 	&=\sum_{k,l,k',l'=0}^{J-1}
 	\F(f)(\chi_k)  \F(g)(\chi_l) \overline{\F(f)(\chi_{k'}) } \,\overline{\F(f)(\chi_{l'}) }D_{N,k,l}\overline{D_{N,k',l'}}\chi_{(k+l)-(k'+l')}(j).
 \end{align}
 Whence
\begin{align}\label{square2}
&\sum_{j=0}^{J-1}\Big|\frac1{N}\sum_{n=1}^{N}\bmnu(n) f(j+n)g(j-n)\Big|^2 \notag \\
&=J \sum_{k+l=k'+l'}^{J-1}
\F(f)(\chi_k)  \F(g)(\chi_l) \overline{\F(f)(\chi_{k'}) } \,\overline{\F(f)(\chi_{l'}) }D_{N,k,l}\overline{D_{N,k',l'}}\chi_{(k+l)-(k'+l')}(j),
\end{align} 	 
since
$$\sum_{j=0}^{J-1}\chi_{m}(j)=\begin{cases}
	0 &\textrm{if $m \neq 0$} ,\\
	J & \textrm{if~not.}
\end{cases}
$$
Therefore,
\begin{align}\label{square5}
&\frac{1}{J} \sum_{j=0}^{J}\Big|\frac1{N}\sum_{n=1}^{N}\bmnu(n) f(j+n)g(j-n)\Big|^2 \\
&=\sum_{k+l=k'+l'}^{J-1}
\F(f)(\chi_k)  \F(g)(\chi_l) \overline{\F(f)(\chi_{k'}) } \,\overline{\F(f)(\chi_{l'}) }D_{N,k,l}\overline{D_{N,k',l'}},\notag\\
&=\sum_{s=0}^{J-1}  \Big|
\sum_{k=0}^{J-1}\F(f)(\chi_k) \F(g)(\chi_{s-k}) D_{N,k,s-k}\Big|^2.\notag
\end{align} 
Now, following Bourgain's approach, for each polynomials $P(n)$, we put 
$$K_{N,P}(k)=\frac{1}{N}\sum_{n=1}^{N}\bmnu(n) \delta_{P(n)}(k),$$
and, we define its off-diagonal by
\begin{align}
 K_{N,P} \odot K_{N,P} = \frac{1}{N}\sum_{n=1}^{N}\bmnu(n)\big( \delta_{P(n)} \otimes \delta_{P(n)}\big)	
\end{align}
where $\otimes$ is the usual product measure define on the Cartesian product $\Z_j\times \Z_j$ by 
$$\F(\mu\otimes\nu)(\chi_k,\chi_l)=\F(\mu)(\chi_k)\F(\nu)(\chi_l).$$
In the similar manner, we define $ K_{N,P} \odot K_{N,Q},$ and more generally if $X=\Big\{x_1,\cdots, x_N\Big\}$ and  $K_N= \frac{1}{N}\sum_{n=1}^{N}a_n \delta_{x_j}$, we put
\begin{align}\label{kernel}
K_N\odot K_N= \frac{1}{N}\sum_{i=1}^{N}a_i\big( \delta_{x_i} \otimes \delta_{x_i}\big).
\end{align}
But for technical convenience, we will use the following off-diagonal kernel
$$L_{N}(k,l)=\frac{1}{N}\sum_{n=1}^{N}\bmnu(n) \big(\delta_{P(n)-Q(n)}\otimes\delta_{Q(n)}\big)(k,l),$$
Therefore
$$D_{N,k,s-k}=\F(L_N)(\chi_k,\chi_s).$$
We are going to establish the following
\begin{align}\label{Blkey}
&\sum_{s=0}^{J-1}  \Big|
\sum_{k=0}^{J-1}\F(f)(\chi_k) \F(g)(\chi_{s-k})\F(L_N)(k,s) \Big|^2 \notag \\
&\leq \frac{C_A}{\log(N)^A}
\Big\|\F(f) * \F(g)\|_2,
\end{align}
which is equivalent to the following version on the torus 
\begin{align}\label{Blkey2}
\sum_{s \in \Z}  \Big|\int_{0}^{1} \widehat{f}(\theta) \widehat{g}(\tau-\theta) 
\widehat{L_N}(\theta, \tau) d\theta e^{-i s \tau} d\tau\Big|^2\notag\\
\leq \frac{C_A}{\log(N)^A}\|f g\|_2^2.
\end{align}
This with the help of Parseval equality is equivalent to  
 \begin{align}\label{Blkey3}
  \int_{0}^{1} \Big|\int_{0}^{1} \widehat{f}(\theta) \widehat{g}(\tau-\theta) 
 \widehat{L_N}(\theta, \tau) d\theta \Big|^2 d\tau \notag\\
 \leq \frac{C_A}{\log(N)^A}\|f g\|_2^2.
 \end{align}
Indeed, it is suffice to consider that $f, g \in \ell^2(\Z)$ are supported on $[0,J]$ and compute in the same manner to derive the formula from
$$  
\frac1{N}\sum_{n=1}^{N}\bmnu(n) f(j+n)g(j-n).
$$
Assume that \eqref{Blkey} is proved. Then, by applying the convolution theorem combined with  Parseval equality, we get
\begin{align}\label{Bl2}
&\frac{1}{J} \sum_{j=0}^{J}\Big|\frac1{N}\sum_{n=1}^{N}\bmnu(n) f(j+n)g(j-n)\Big|^2 \notag
\\
&\leq \Big\|fg\Big\|_2^2  ,
\end{align}
Therefore, by applying Cauchy-Schwarz inequality, we  get 
\begin{align}
\frac{1}{J}\sum_{j=0}^{J}\Big|\frac1{N}\sum_{n=1}^{N}\bmnu(n) f(j+n)g(j-n)\Big|^2 \notag\\
\leq \frac{C_A^2}{{\big(\log(N)\big)^{2A}}} \big\|f\big\|_4^2 \big\|g\big\|_4^2
\end{align}
Whence 
\begin{align}
\Big\|\frac1{N}\sum_{n=1}^{N}\bmnu(n) f(j+n)g(j-n)\Big\|_2  
\leq \frac{C_A}{{\big(\log(N)\big)^{A}}} \big\|f\big\|_4 \big\|g\big\|_4
\end{align}
At this point, we need only to establish \eqref{Blkey} or equivalently \eqref{Blkey2}.
For that, we use the  Bourgain's approach of the circle method (see \cite[Chap. 5]{Vaughan}, \cite{Wierdl}, \cite{Cuny-Weber}, \cite{Nair} \footnote{But, as noticed in \cite{Cuny-Weber} there is a small gap in Wierdl's argument which is corrected in \cite[Lemma 6.2]{Cuny-Weber}. We stress also that Young's inequality for convolution play a crucial role in \cite[Proposition 7.2]{Cuny-Weber}, and in our case, we need to deals with the approximation of $(\theta,\tau) \in [0,1)^2.$} to get even the following  
\begin{align}\label{BNWCW}
\sum_{j \in \Z}\sup_{ N  \geq 1}\Bigg|\int_{0}^{1} \Bigg(\int_{0}^{1} \widehat{f}(\theta) \widehat{g}(\tau-\theta) 
\widehat{L_N}(\theta, \tau) d\theta e^{-ij\tau}\Bigg)  d\tau\Bigg|^2 \notag\\
\leq C\|f g\|_2^2.
\end{align}


Now, by applying the same reasoning as in \cite{IJL}, we conclude that 
for any $K \geq 1$, we have
\begin{eqnarray}\label{BourgainMaximal2} 
\sum_{k=1}^{K}
\Big\|m_{N_k,N_{k+1}}(f,g)\Big\|_{\ell^2(\Z)} < C. \sqrt{K}.
\big\|f\big\|_{\ell^4(\Z)}\big\|g\big\|_{\ell^4(\Z)}.
\end{eqnarray}
This complete the proof of the theorem.
\end{proof}
\noindent As a consequence of Proposition \ref{CalderonP} and \ref{CalderonP2}, we have the following theorem. Its proof is similar to the proof of Propositions \ref{CalderonI} and \ref{CalderonP}. But, we provide it for the reader's convenience. 

\begin{thm}\label{MaximalI}Let $(X,\A,T,\mu)$ be an ergodic dynamical system, and let $f,g \in L^4(X,\mu)$ . Then, 
	for any $\rho>1$, 
	for any $K \geq 1$,
	\begin{eqnarray}\label{BourgainMaximal3}
	\sum_{k=1}^{K}\Big\|M_{N_k,N_{k+1}}(f,g)\Big|\Big\|_1< 4C .\sqrt{K}.\big\|f\big\|_{4}\big\|g\big\|_{4},
	\end{eqnarray}
	where $\bmnu$ is the Liouville function or the M\"{o}bius function.
\end{thm}
\begin{proof}
	Let $\bar{N} = N_{K+1}$ and $J \gg \bar{N}$.  Put
	$$\phi_x(n)=\left\{
	\begin{array}{ll}
	f(T^{P(n)}x), & \hbox{if $n \in [-2\bar{N},2\bar{N}]$;} \\
	0, & \hbox{if not,}
	\end{array}
	\right.$$
	and
	$$\psi_x(n)=\left\{
	\begin{array}{ll}
	g(T^{Q(n)}x), & \hbox{if $n \in [-2\bar{N},2\bar{N}]$;} \\
	0, & \hbox{if not,}
	\end{array}
	\right.
	$$
	Then, by Theorem \ref{CalderonI}, we have
	\begin{eqnarray*}
		\sum_{k=1}^{K}\Big\|
		m_{N_k,N_{k+1}}(\phi_x,\psi_x)\Big\|_{\ell^2(\Z)}<C \sqrt{K}   \big\|\phi_x\big\|_{\ell^4(\Z)} \big\|\psi_x\big\|_{\ell^4(\Z)}.
	\end{eqnarray*}
	We thus get
	$$\sum_{k=1}^{K} \Big(
	\sum_{|j| \leq \bar{N}} \Big(m_{N_k,N_{k+1}}(\phi_x,\psi_x)(j)\Big)^2\Big)^{\frac12}< C \sqrt{K} \big\|\phi_x\big\|_{\ell^4(\Z)} \big\|\psi_x\big\|_{\ell^4(\Z)},$$
	which can be rewritten as follows
	\begin{align*}
		&\sum_{k=1}^{K} \Big(\sum_{|j| \leq \bar{N}} \Big(M_{N_k,N_{k+1}}(f,g)(T^jx)\Big)^2\Big)^{\frac12} \\
		&< C \sqrt{K}\Big(\sum_{|n| \leq 2 \bar{N}}|f|^4(T^nx)\Big)^{\frac14}
		\Big(\sum_{|n| \leq 2 \bar{N}}|g|^4(T^nx)\Big)^{\frac14} . 
	\end{align*}
	Integrating and applying H\"{o}lder inequality we obtain
	$$\sum_{k=1}^{K}\Big\|M_{N_k,N_{k+1}}(f,g)\Big\|_1<4 C \sqrt{K}\big\|f\big\|_4 \big\|g\big\|_4,$$
	since $T$ is measure preserving, and this finish the proof of the theorem.
\end{proof}

\noindent We proceed now to the proof of our main result (Theorem \ref{Mainofmain}).
\begin{proof}[\textbf{Proof of Theorem \ref{Mainofmain}}.]Without loss of generality, we assume that the map $T$ is totally ergodic, that is, all its powers are ergodic. Let us assume also that $f,g$ are in $L^\infty(X,\mu)$. Therefore, by Theorem \ref{MaximalI}, it is easily seen that 
\begin{align}\label{Key-conv}
\frac{1}{K}\sum_{k=1}^{K}\Big\|M_{N_k,N_{k+1}}(f,g)\Big\|_1 \tend{K}{+\infty}0.
\end{align}
	\noindent Hence, by the standard arguments of oscillation method (see for instance \cite{Nair}, \cite{Thouvenot}, \cite{Demeter}), for almost every point $x \in X$, we have
	\begin{align*}
	\frac1{[\rho^m]}\sum_{n=1}^{[\rho^m]}\bmnu(n)f(T^{P(n)}x)g(T^{Q(n)}x)  \tend{m}{+\infty}0,
	\end{align*}
	since the $L^2$-limit is zero by Green-Tao theorem \cite[Theorem 1.1]{GreenT}  combined with Chu's result  \cite[Theorem 1.3]{Chu}.
	
	\noindent For the reader's convenience, let us point out that the proof in \cite{Thouvenot} and \cite{Demeter} is obtained by contradiction. Indeed, we assume that the almost everywhere convergence does not hold. Then, we construct an increasing sequence $(N_k)$ for which we establish with the help of the Markov trick that \eqref{Key-conv} can not hold.\\	
	Now, it follows that if $[\rho^m]\leq N < {[\rho^{m+1}]+1}$,  
	\begin{eqnarray*}
		&&\Big|\frac1{N}\sum_{n=1}^{N}\bmnu(n)f(T^{P(n)}x)g(T^{Q(n)}x) \Big|\\
		&=&\Big| \frac1{N}\sum_{n=1}^{[\rho^m]}\bmnu(n)f(T^{P(n)}x)g(T^{Q(n)}x) +\\&& \frac1{N}\sum_{n=[\rho^m]+1}^{N}\bmnu(n)f(T^{P(n)}x)g(T^{Q(n)}x)\Big|\\
		&\leq & \Big| \frac1{[\rho^m]}\sum_{n=1}^{[\rho^m]}\bmnu(n)f(T^{P(n)}x)g(T^{Q(n)}x) \Big|+\\&&\frac{\big\|f\big\|_{\infty} \big\|g\big\|_{\infty}}{[\rho^m]} (N-[\rho^m]-1)\\
		&\leq & \Big|\frac1{[\rho^m]}\sum_{n=1}^{[\rho^m]}\bmnu(n)f(T^{P(n)}x)g(T^{Q(n)}x)\Big|+\\&&\frac{\big\|f\big\|_{\infty} \big\|g\big\|_{\infty}}{[\rho^m]} ([\rho^{m+1}]-[\rho^m]).\\
	\end{eqnarray*}
	Letting $m$ goes to infinity, we get
	\[
	\Big| \frac1{[\rho^m}\sum_{n=1}^{[\rho^m]}\bmnu(n) f(T^{P(n)}x)g(T^{Q(n)}x) \Big| \tend{m}{+\infty}0,
	\]
	and
	\[
	\frac{||f||_{\infty}\big\|g\big\|_{\infty}}{[\rho^m]} ([\rho^{m+1}]-[\rho^m])\tend{m}{+\infty}\big\|f\big\|_{\infty} \big\|g\big\|_{\infty}.(\rho-1),
	\]
	for any $\rho>1$.  Letting $\rho \longrightarrow 1$ we conclude that
	\[
	\frac1{N}\sum_{n=1}^{N}\bmnu(n)f(T^{P(n)}x)g(T^{Q(n)}x) \tend{N}{+\infty}0, {\textrm{~~a.e.,}}
	\]
	
	To finish the proof, notice that for any $f,g \in L^2(X,\mu)$, and any $\varepsilon>0$, there exist $f_1,g_1 \in L^{\infty}(X,\mu)$ such that
	$\Big\|f-f_1\Big\|_2 < \sqrt{\varepsilon},$ and $\Big\|g-g_1\Big\|_2 < \sqrt{\varepsilon}$.
	Moreover, by Cauchy-Schwarz inequality, we have
	\begin{eqnarray*}
		&&\Big|\frac1{N}\sum_{n=1}^{N}\bmnu(n)(f-f_1)(T^{P(n)}x)(g-g_1)(T^{Q(n)}x) \Big|\\
		&\leq&  \frac1{N}\sum_{n=1}^{N}\big|(f-f_1)(T^{P(n)}x)\big|\big|(g-g_1)(T^{Q(n)}x)\big|\\
		&\leq& \Big(\frac1{N}\sum_{n=1}^{N}\big|(f-f_1)(T^{P(n)}x)\big|^2\Big)^{\frac12}
		\Big(\frac1{N}\sum_{n=1}^{N}\big|(g-g_1)(T^{Q(n)}x)\big|^2\Big)^{\frac12}
	\end{eqnarray*}
	Applying the ergodic theorem, it follows that for almost all $x \in X$, we have
	\[
	\limsup_{N \longrightarrow +\infty}\Big|\frac1{N}\sum_{n=1}^{N}\bmnu(n)(f-f_1)(T^P(n)x)(g-g_1)(T^{Q(n)}x) \Big|<  \varepsilon.
	\]
	Whence, we can write
	\begin{eqnarray*}
		&&\limsup_{N \longrightarrow +\infty}\Big|\frac1{N}\sum_{n=1}^{N}\bmnu(n)f(T^{P(n)}x)g(T^{Q(n)}x) \Big|\\
		&\leq&
		\limsup_{N \longrightarrow +\infty}\Big|\frac1{N}\sum_{n=1}^{N}
		\bmnu(n)f_1(T^{P(n)}x)g(T^{Q(n)}x)\Big|\\
		&+&
		\limsup_{N \longrightarrow +\infty}\Big|\frac1{N}\sum_{n=1}^{N}\bmnu(n)f(T^{P(n)}x)g_1(T^{Q(n)}x)\Big|\\
		&+&\limsup_{N \longrightarrow +\infty}\Big|\frac1{N}\sum_{n=1}^{N}\bmnu(n)f_1(T^{P(n)}x)g_1(T^{Q(n)}x)\Big|\\
		&\leq& \varepsilon+\limsup_{N \longrightarrow +\infty}\Big|\frac1{N}\sum_{n=1}^{N}
		\bmnu(n)f_1(T^{P(n)}x)g(T^{Q(n)}x)\Big|\\
		&+&
		\limsup_{N \longrightarrow +\infty}\Big|\frac1{N}\sum_{n=1}^{N}\bmnu(n)f(T^{P(n)}x)g_1(T^{Q(n)}x)\Big|.
	\end{eqnarray*}
	We thus need to estimate
	$$\limsup_{N \longrightarrow +\infty}\Big|\frac1{N}\sum_{n=1}^{N}
	\bmnu(n)f_1(T^{P(n)}x)g(T^{Q(n)}x)\Big|,$$
	and
	$$
	\limsup_{N \longrightarrow +\infty}\Big|\frac1{N}\sum_{n=1}^{N}\bmnu(n)f(T^{P(n)}x)g_1(T^{Q(n)}x)\Big|.$$
	In the same manner we can see that
	\begin{eqnarray*}
		&&\limsup_{N \longrightarrow +\infty}\Big|\frac1{N}\sum_{n=1}^{N}
		\bmnu(n)f_1(T^nx)(g-g_1)(T^{Q(n)}x)\Big|\\
		&\leq& \limsup_{N \longrightarrow +\infty}\Big(\frac1{N}\sum_{n=1}^{N}|f_1(T^{P(n)}x)|^2\Big)^{\frac12}
		\limsup_{N \longrightarrow +\infty}\Big(\frac1{N}\sum_{n=1}^{N}|(g-g_1)(T^{Q(n)}x)|^2\Big)^{\frac12}\\
		&\leq& \big\|f_1\big\|_2 \big\|g-g_1\big\|_2\\
		&\leq& \Big(\big\|f\big\|_2+\sqrt{\varepsilon}\Big). \sqrt{\varepsilon}
	\end{eqnarray*}
	This gives
	\begin{eqnarray*}
		&&\limsup_{N \longrightarrow +\infty}\Big|\frac1{N}\sum_{n=1}^{N}
		\bmnu(n)f_1(T^{P(n)}x)g(T^{Q(n)}x)\Big|\\
		&\leq& \Big(\Big\|f\Big\|_2+\sqrt{\varepsilon}\Big). \sqrt{\varepsilon}+
		\limsup_{N \longrightarrow +\infty}\Big|\frac1{N}\sum_{n=1}^{N}
		\bmnu(n)f_1(T^{P(n)}x)g_1(T^{Q(n)}x)\Big|\\
		&\leq& \Big(\big\|f\big\|_2+\sqrt{\varepsilon}\Big). \sqrt{\varepsilon}+0
	\end{eqnarray*}
	Summarizing, we obtain the following estimates
	\begin{eqnarray*}
		&&\limsup_{N \longrightarrow +\infty}\Big|\frac1{N}\sum_{n=1}^{N}\bmnu(n)f(T^{P(n)}x)g(T^{Q(n)}x) \Big|\\
		&\leq& \varepsilon+\Big(\big\|f\big\|_2+\sqrt{\varepsilon}\Big). \sqrt{\varepsilon}+\Big(\big\|g\big\|_2+\sqrt{\varepsilon}\Big). \sqrt{\varepsilon}
	\end{eqnarray*}
	Since $\varepsilon>0$ is arbitrary, we conclude that for almost every $x \in X$,
	\[
	\frac1{N}\sum_{n=1}^{N}\bmnu(n)f(T^{P(n)}x)g(T^{Q(n)}x) \tend{N}{+\infty}0.
	\]
	This complete the proof of the theorem.
\end{proof}
It is noticed in \cite{IJL} that the convergence almost sure holds for the short interval can be obtained by applying the following Zhan's estimation \cite{Ztao}: for each $A>0$, for any $\varepsilon>0$, we have
\begin{equation}\label{Zhantao}
\max_{z \in \T}\left|\displaystyle\sum_{ N \leq n \leq N+M}z^n\bml(n)\right|\leq C_{A,\varepsilon}\frac{M}{\log^{A}(M)}\text{ for some }C_{A,\varepsilon}>0,
\end{equation}
provided that $M \geq N^{\frac58+\varepsilon}$.  Here, 
\begin{xques}we ask on the convergence almost sure in the short interval for the polynomial  bilinear ergodic bilinear averages 
with Liouville and M\"{o}bius weights.
\end{xques}
\begin{rem}In the forthcoming revised version of \cite{elabdal-poly}. The authors will present a simple proof of Bourgain double ergodic theorem by applying the strategy presented here and by adapting it to the proof of Bourgain bilinear ergodic
	theorem for polynomials and polynomials
	in primes as it is stated in that paper.\footnote{The paper \cite{elabdal-poly} was posted on Arxiv on August 2019. Therein,  it is pointed that Babenko-Beckner inequalities will be used in the proof of Bourgain bilinear ergodic theorem for polynomials and polynomials
		in primes in the forthcoming revision of the paper. Very recently (3 August 2020), the authors in \cite{TMT} provided a proof of a partial Bourgain bilinear ergodic
		theorem for polynomials  for $\sigma$-finite measure space. It is seems that therein, the authors extended some ideas introduced by the author based on the Cald\'eron principal combined with the duality argument due to F. Riesz (precisely, $\ell^1$-$\ell^\infty$-duality). A part of the plan presented in \cite{TMT} seems to be in common with that in \cite{elabdal-poly} well be addressed partially in the revised version of \cite{elabdal-poly}.} 
\end{rem}

\end{document}